\documentclass{birkjour}
 \newtheorem{thm}{Theorem}[section]
 \newtheorem{cor}[thm]{Corollary}
 \newtheorem{lem}[thm]{Lemma}
 
 \theoremstyle{definition}
 \newtheorem{defn}[thm]{Definition}
 \theoremstyle{remark}
 \newtheorem{rem}[thm]{Remark}
 \newtheorem*{ex}{Example}
 \numberwithin{equation}{section}

\begin{document}

%
%
%
%
%
%
%
%
%

\title[An effective metric on $C(H,K)$ with normal structure]
 {An effective metric on $C(H,K)$ with normal structure}

\author[Mona Nabiei]{Mona Nabiei}

\address{%
Department of Mathematics, \\
Shahid Beheshti University, G. C.\\
P.O. Box 19839\\
Tehran, IRAN}

\email{m$_-$nabiei@sbu.ac.ir}



\date{Apr. 3, 2014}

\begin{abstract}
This study first defines a new metric with normal structure on $C(H,K)$ and then a new technique to prove fixed point theorems for families of non-expansive maps on this metric space.
Indeed, it shows that the presence of a bounded orbit implies the existence of a fixed point for a group of h-biholomorphic automorphisms on $C(H,K)$.  
\end{abstract}

\maketitle
\section{Introduction}
The theory of bounded operators on a Hilbert space is an important branch of mathematics; however, all operators that arise naturally are not bounded. It is important to study unbounded operators that provide abstract frameworks for dealing with differential operators and unbounded observables in quantum mechanics, among others. 

The Kobayashi metric play a crucial role in the study of holomorphic maps on open unit ball of any complex Banach space \cite{fran}. One feature of this metric is its invariability for holomorphic maps. The present study defines a metric on $C(H,K)$ with similar properties.

The paper defines a suitable metric $d$ on $C(H,K)$, the set of all closed densely-defined linear operators from Hilbert space $H$ to finite dimensional Hilbert space $K$. Metric geometric properties, such as compactness of the class of admissible sets, metric convexity, and normal structure are examined. Then it is shown that the presence of a bounded orbit implies the existence of an invariant admissible set.

The h-biholomorphic maps are introduced at the end of this paper. Then, it is shown that the presence of a bounded orbit implies the existence of a fixed point for a group of h-biholomorphic automorphisms on $C(H,K)$.
\section{Preliminary Notes}
Let $X$ be a complex Banach space and let $B$ be the open unit ball in $X$. Let the Poincar\'e metric $\omega$ on $\Delta$, the open unit disc in the complex plane $\mathbb{C}$, be given by
$$\omega(a,b)=\tanh^{-1}\frac{|a-b|}{|1-\bar{a}b|} \ \ \ \ \ |a|, |b|<1.$$

Let $x,y$ be two points of $B$. An analytic chain joining $x$ and $y$ in $B$ consists of $2n$ points $z'_1,z''_1,...,z'_n,z''_n$ in $\Delta$, and of $n$ holomorphic functions $f_k:\Delta\rightarrow B$, such that 
$$f_1(z'_1)=x, \dots , f_k(z''_k)=f_{k+1}(z'_{k+1}) \ \ \ \ for~ k=1, \dots , n-1,\ \ f_n(z''_n)=y.$$
Since $B$ is connected, given $x$ and $y$, an analytic chain joining $x$ and $y$ in $B$ always exists, provided that $n$ is sufficiently large. Let
$$K_B(x,y)=\inf\{\omega(z'_1,z''_1)+\omega(z'_2,z''_2)+\dots +\omega(z'_n,z''_n)\},$$
where the infimum is taken over all choices of analytic chains joining $x$ and $y$ in $B$. It is called the Kobayashi pseudo-metric on $B$. 

The next theorems from \cite{fran} will be used later on.

\begin{thm}\label{iso}
Let $B$ be the open unit ball of the complex Banach space $X$. If $F:B\rightarrow B$ is a bi-holomorphic map, then
$$K_B(F(x),F(y))=K_B(x,y),\ \ \ for~all~x,y\in B.$$
\end{thm}

\begin{thm}\label{all metric}
Let $B$ be the open unit ball of a complex Banach space $X$. Then
$$K_B(0,x)=\omega(0,\|x\|).$$
\end{thm}

Let $H$, $K$ be complex Hilbert spaces. We denote the open unit ball of $L(K,H)$, the space of all bounded linear operators from $K$ to $H$, by $\mathcal{B}$.   For each $A \in\mathcal{B}$, we define a transformation $\eta$ on $\mathcal{B}$ setting
\begin{equation}\label{linear transformation}
\eta(Z)=(I-AA^*)^{-\frac{1}{2}}(Z+A)(I+A^*Z)^{-1}(I-A^*A)^{\frac{1}{2}}.
\end{equation}
We collect the facts about this  transformation that we need in the following lemma.
\begin{lem}\label{llt} If $\eta$ is as in (\ref{linear transformation}), then $\eta$ has the following properties:

\begin{itemize}
\item [(i)]$\eta$ is invertible and its inverse is given by
$$\eta^{-1}(Z)=(I-AA^*)^{-\frac{1}{2}}(Z-A)(I-A^*Z)^{-1}(I-A^*A)^{\frac{1}{2}}.$$
\item[(ii)] $\eta$ is a biholomorphic mapping on $\mathcal{B}$.
\item[(iii)] If $\dim(K)<\infty$, then $\eta$ is WOT-continuous.
\end{itemize}
\end{lem}
\begin{proof}
To prove statement (i), take any $Z\in\mathcal{B}$ and insert
$$Y=(I-AA^*)^{-\frac{1}{2}}(Z-A)(I-A^*Z)^{-1}(I-A^*A)^{\frac{1}{2}}.$$
We shall show  that $\eta(Y)=Z$. To do this we should note that
\begin{equation}\label{commut}
Y=(I-AA^*)^{\frac{1}{2}}(I-ZA^*)^{-1}(Z-A)(I-A^*A)^{-\frac{1}{2}}.
\end{equation}
For, (\ref{commut}) is equivalent to the identity
$$(Z-A)(I-A^*A)^{-1}(I-A^*Z)=(I-ZA^*)(I-AA^*)^{-1}(Z-A),$$
which can be easily verified by a straightforward calculation. 
Now, we evaluate the operator $I+A^*Y$:
\begin{eqnarray}
I+A^*Y&=&I+A^*(I-AA^*)^{\frac{1}{2}}(I-ZA^*)^{-1}(Z-A)(I-A^*A)^{-\frac{1}{2}}\nonumber\\
~&=&I+(I-A^*A)^{\frac{1}{2}}(I-A^*Z)^{-1}(A^*Z-A^*A)(I-A^*A)^{-\frac{1}{2}}\nonumber\\
~&=&I-I+(I-A^*A)^{\frac{1}{2}}(I-A^*Z)^{-1}(I-A^*A)^{\frac{1}{2}}\nonumber\\
~&=&(I-A^*A)^{\frac{1}{2}}(I-A^*Z)^{-1}(I-A^*A)^{\frac{1}{2}}.\nonumber
\end{eqnarray} 
From this equation we see at once that
\begin{eqnarray}
\eta(Y)&=&(I-AA^*)^{-\frac{1}{2}}(Y+A)(I-A^*A)^{-\frac{1}{2}}(I-A^*Z)\nonumber\\
~&=&[Z(I-A^*Z)^{-1}-ZA^*(I-ZA^*)^{-1}A](I-A^*A)^{-1}(I-A^*Z)\nonumber\\
~&=&Z(I-A^*Z)^{-1}[I-A^*A](I-A^*A)^{-1}(I-A^*Z)\nonumber\\
~&=&Z.\nonumber
\end{eqnarray}

Statement (ii) was shown in theorem 2 of \cite{harr}.
To prove statement (iii): It is easy to check that if $K$ is finite dimensional, then the norm topology of $L(K,H)$ coincides with the strong operator topology while WOT coincides with the weak topology of $L(K,H)$. In this case, the WOT-continuity of $\eta$ was noticed and used by Krein \cite{krei}. 
\end{proof}

Now, we denote the space of all closed densely-defined linear operators from $H$ to $K$ by $C(H,K)$. The first relaxation in the concept of operator is to not assume that the operators are defined everywhere on $H$. Hence,  densely-defined operator $T:H\rightarrow K$ is a linear function whose domain of definition is dense linear subspace $\mathcal{D}(T)$  in $H$. 
$T$ is closed if its graph, $\mathcal{G}(T)$, is a closed subset of space $H\oplus K$ \cite{schm}.

Let $T\in C(H,K)$ and define $L_T$ and $R_T$ settings as: 
$$L_T(X)=(I+T^*T)^{\frac{1}{2}}X^*-T^*(I+XX^*)^\frac{1}{2},$$
$$R_T(X)=(I+XX^*)^\frac{1}{2}(I+TT^*)^\frac{1}{2}-XT^*.$$

Consider operator $X$ such that the compositions are well-defined. Using lemma 1.10 of Schm\"udgen \cite{schm}, recall that $\mathcal{G}(T^*)=V(\mathcal{G}(T))^{\perp}$ where $V(x,y)=(-y,x)$, $x\in H$, $y\in K$. Hence, $K\bigoplus H=\mathcal{G}(T^*)\bigoplus V(\mathcal{G}(T))$. Therefore, for each $u\in H$, there exist $x\in\mathcal{D}(T)$ and $y\in\mathcal{D}(T^*)$ such that $y-Tx=0$, $T^*y+x=u$. That is, $I+T^*T$ is surjective.  $T^*T$ is a positive self-adjoint operator and, for $x\in \mathcal{D}(T^*T)$: 
$$\|(I+T^*T)x\|^2=\|x\|^2+\|T^*Tx\|^2+2\|Tx\|^2,$$
hence, $I+T^*T$ is a bijective mapping with a positive bounded self-adjoint inverse on $H$ such that:
$$0\leq (I+T^*T)^{-1}\leq I.$$  

On the other hand: 
$$(I+T^*T)^{-1}(H)=\mathcal{D}(I+T^*T)=\mathcal{D}(T^*T),$$
hence:
\begin{eqnarray}
\|T(I+T^*T)^{-1}x\|^2 &=& \langle T^*T(I+T^*T)^{-1}x,(I+T^*T)^{-1}x\rangle  \nonumber\\
~&\leq & \langle T^*T(I+T^*T)^{-1}x,(I+T^*T)^{-1}x\rangle \nonumber\\
~&~& \ \ \ \   +  \langle (I+T^*T)^{-1}x,(I+T^*T)^{-1}x\rangle \nonumber\\
~&= &  \langle(I+T^*T)(I+T^*T)^{-1}x,(I+T^*T)^{-1}x\rangle \nonumber\\
~&=&\langle x,(I+T^*T)^{-1}x\rangle =\|(I+T^*T)^{-\frac{1}{2}}x\|^2,  \nonumber
\end{eqnarray}
that is:\\
\centerline{$\|T(I+T^*T)^{-\frac{1}{2}}y\|\leq\|y\|$\ \  \ for $y\in (I+T^*T)^{-\frac{1}{2}}(H)$.}
Since $(I+T^*T)^{-\frac{1}{2}}$ is bijective, $(I+T^*T)^{-\frac{1}{2}}H$ is dense in $H$. Operator $T(I+T^*T)^{-\frac{1}{2}}$ is closed since $T$ is closed and $(I+T^*T)^{-\frac{1}{2}}$ is bounded \cite{schm}. This implies that $\mathcal{D}(T(I+T^*T)^{-\frac{1}{2}})=H$, and $\|T(I+T^*T)^{-\frac{1}{2}}\|\leq1$. A similar argument shows that $\|(I+T^*T)^{-\frac{1}{2}}T^*\|\leq 1$; however, if $K$ is finite dimensional, operator $TT^*$ is bounded and 
$$\|T(I+T^*T)^{-\frac{1}{2}}\|^2=\|TT^*(I+TT^*)^{-1}\|=\frac{\|TT^*\|}{1+\|TT^*\|}<1.$$

It is known that $\|f(TT^*)\|=f(\|TT^*\|)$ if function $f$ is non-decreasing on the interval $[0,\|TT^*\|]$ and $TT^*$ is a positive operator; thus, if $K$ is a finite dimensional Hilbert space, the inverse of the operator $R_S(T)$ exists, and: 
$$R_S(T)^{-1}=(I+SS^*)^{-\frac{1}{2}}[I-T(I+T^*T)^{-\frac{1}{2}}(I+S^*S)^{-\frac{1}{2}}S^*]^{-1}(I+TT^*)^{-\frac{1}{2}}.$$ 

\begin{rem}
From the above it follows that if $K$ is finite dimensional and $T\in C(H,K)$, then $\hat{T}=(I+T^*T)^{-\frac{1}{2}}T^*\in\mathcal{B}$. If $A\in\mathcal{B}$, then $\ker(I-A^*A)={0}$. Because, if $A^*A(x)=x$ then 
$$\langle A^*A(x),x\rangle=\langle x,x\rangle,$$
which meanse $\|A(x)\|=\|x\|$ and $x$ must be zero. So, $A_0=(I-A^*A)^{-\frac{1}{2}}A^*$ is a closed densely-defined operator from $H$ to $K$, such that $\hat{A_0}=A$.
\end{rem}

\begin{lem}\label{mobius}
If  $T\in C(H,K)$, then $\psi_T$, by the following definition, is of the form of \ref{linear transformation}, $\psi_T(\hat{T})=0$, $\psi_T(0)=-\hat{T}$ and $\psi^{-1}_T=\psi_{-T}$.\\ For each $X\in \mathcal{B}$,  
$\psi_T(X)=L_T(Y)R_Y(T)^{-1}$ where $Y=(I-X^*X)^{-\frac{1}{2}}X^*$.
\end{lem}
\begin{proof}
It is easy to see that $\psi_T$ is of the form of \ref{linear transformation}, because:

\leftline{$\psi_T(X)=L_T(Y)R_Y(T)^{-1}$}
\leftline{$=\{(I+T^*T)^\frac{1}{2}Y^*-T^*(I+YY^*)^\frac{1}{2}\}\{(I+TT^*)^\frac{1}{2}(I+YY^*)^\frac{1}{2}-TY^*\}^{-1}$ }
\leftline{$=\{(I+T^*T)^\frac{1}{2}Y^*-T^*(I+YY^*)^\frac{1}{2}\}(I+YY^*)^{-\frac{1}{2}}\{(I+TT^*)^\frac{1}{2}-T\hat{Y}\}^{-1}$ }
\leftline{$=\{(I+T^*T)^\frac{1}{2}\hat{Y}-T^*\}\{(I+TT^*)^\frac{1}{2}-T\hat{Y}\}^{-1}$ }
\leftline{$=\{[(I+T^*T)^{-1}(I+T^*T-T^*T)]^{-\frac{1}{2}}\hat{Y}-T^*\}\{(I+TT^*)^\frac{1}{2}-T\hat{Y}\}^{-1}$ }
\leftline{$=\{[I-\hat{T}\hat{T}^*]^{-\frac{1}{2}}\hat{Y}-T^*\}\{(I+TT^*)^\frac{1}{2}-T\hat{Y}\}^{-1}$ }
\leftline{$=\{(I-\hat{T}\hat{T}^*)^{-\frac{1}{2}}\hat{Y}-[I+T^*T-T^*T]^{\frac{1}{2}}T^*\}\{(I+TT^*)^\frac{1}{2}-T\hat{Y}\}^{-1}$ }
\leftline{$=\{(I-\hat{T}\hat{T}^*)^{-\frac{1}{2}}\hat{Y}-(I-\hat{T}\hat{T}^*)^{-\frac{1}{2}}\hat{T}\}\{(I+TT^*)^\frac{1}{2}-T\hat{Y}\}^{-1}$ }
\leftline{$=(I-\hat{T}\hat{T}^*)^{-\frac{1}{2}}(\hat{Y}-\hat{T})\{(I+TT^*)^\frac{1}{2}-T\hat{Y}\}^{-1}$ }
\leftline{$=(I-\hat{T}\hat{T}^*)^{-\frac{1}{2}}(\hat{Y}-\hat{T})\{I-(I+TT^*)^{-\frac{1}{2}}T\hat{Y}\}^{-1}(I+TT^*)^{-\frac{1}{2}}$ }
\leftline{$=(I-\hat{T}\hat{T}^*)^{-\frac{1}{2}}(\hat{Y}-\hat{T})\{I-\hat{T}^*\hat{Y}\}^{-1}(I+TT^*)^{-\frac{1}{2}}$ }
\leftline{$=(I-\hat{T}\hat{T}^*)^{-\frac{1}{2}}(\hat{Y}-\hat{T})(I-\hat{T}^*\hat{Y})^{-1}\{(I+TT^*)^{-1}[I+TT^*-TT^*]\}^\frac{1}{2}$ }
\leftline{$=(I-\hat{T}\hat{T}^*)^{-\frac{1}{2}}(\hat{Y}-\hat{T})(I-\hat{T}^*\hat{Y})^{-1}\{I-\hat{T}^*\hat{T}\}^\frac{1}{2}$ }
\leftline{$=(1-\hat{T}\hat{T}^*)^{-\frac{1}{2}}(X-\hat{T})(I-\hat{T}^*X)^{-1}(I-\hat{T}^*\hat{T})^{\frac{1}{2}},$ }

\leftline{where $\hat{T}=(1+T^*T)^{-\frac{1}{2}}T^*\in \mathcal{B}$.}
\leftline{Therefore, by lemma \ref{llt}, $\psi^{-1}_T=\psi_{-T}$, $\psi_T(\hat{T})=0$ and $\psi_T(0)=-\hat{T}$.} 
\end{proof}

If $H$, $K$ are complex Hilbert spaces and $\dim K<\infty$,  set:
$$d(T,S)=\tanh ^{-1}\|L_T(S)R_{S}(T)^{-1}\|.$$
It is easy to see that, in this case, by theorems \ref{iso}, \ref{all metric} and  lemma \ref{mobius}, $d$ defines a metric on $C(H,K)$. This metric satisfies the next equalities: 
\begin{eqnarray}
d(T,S)&=&\tanh^{-1}\|L_T(S)R_S(T)^{-1}\|=\omega(0,\|\psi_T(\hat{S})\|)\nonumber\\
~&=&K_{\mathcal{B}}(0,\psi_T(\hat{S}))=K_{\mathcal{B}}(\psi_T(\hat{T}),\psi_T(\hat{S}))=K_{\mathcal{B}}(\hat{T},\hat{S}),\nonumber
\end{eqnarray}
where $\psi_T$ is as in lemma \ref{mobius}.

\section{The metric geometry properties} 
From here to the end of the article, $K$ is finite dimensionl Hilbert space, and by $C(H,K)$ we mean that the metric space $(C(H,K),d)$. A metric space $(M,\rho)$ is said to be metrically convex if given any two points $p,q\in M$ there exists at least one point $r\in M$ such that $r$ is metrically between $p$ and $q$ i.e. $d(p,q)=d(p,r)+d(r,q)$. Metric convexity is a fundamental concept in the axiomatic study of the geometry of metric spaces. However, in its most general form, it fail to satisfy one of the basic properties of convexity in the algebraic sense (in a linear space); namely, if $A$ and $B$ are two metrically convex subsets of a metric space, then it need not be the case that $A\cap B$ is metrically convex. To see this it is sufficient only to consider the ordinary unit circle $S$ in plane where the distance between two points is taken to be the length of the shortest arc joining them. The upper half circle and lower half circle of $S$ are each metrically convex but the intersection of these two sets is easily seen to be a set consisting of just two antipodal points. 

A subset $D$ of a metric space $M$ will be said to be admissible if $D$ can be written as the intersection of a family of closed balls  centered at points of $M$. The family $\mathcal{A}(M)$ of all admissible subsets of $M$, enters into the study of metric fixed point theory in a very natural way, is therefore the obvious candidate for the needed underlying convexity structure. The intersection of two admissible subset is admissible, and it is a very important property of these sets. For more details, you can see \cite{horv, kham}. In this section we will show that  $(C(H,K),d)$ is metrically convex and each admissible subset of this metric space is metrically convex, too.  Before this, we consider the following. 

\begin{lem}\label{phi}
For every $T,S\in C(H,K)$, there is a biholomorphic automorphism $\varphi$ on $\mathcal{B}$ with $\varphi(\hat{T})=-\varphi(\hat{S})$.  
\end{lem}
\begin{proof}  Let $B=\psi_{-T}(-\hat{S})$, and $B=V|B|$ be the polar decomposition of $B$ (so $|B|=(B^*B)^{1/2}$). Put $C=\tanh (\frac{1}{2}\tanh ^{-1}(|B|))$, and $A=VC$.
By the below calculation, for $A_0=(1-A^*A)^{-\frac{1}{2}}A^*$, we have  $B=\psi_{-A_0}(A)$.
\begin{eqnarray}
\psi_{-A_0}(A)&=& L_{-A_0}(A_0)R_{A_0}(-A_0)^{-1}\nonumber\\
~&=&2A_0^*[I+A_0(I+A_0^*A_0)^{-1}A_0^*]^{-1}(I+A_0A_0^*)^{-\frac{1}{2}} \nonumber \\
~&=&2A(I-A^*A)^{-\frac{1}{2}}[I+A_0(I-AA^*)A^*_0]^{-1}(I-A^*A)^{\frac{1}{2}}\nonumber\\
~&=&2A(I-A^*A)^{-\frac{1}{2}}(I+A^*A)^{-1}(I-A^*A)^{\frac{1}{2}}\nonumber\\
~&=&2VC (I-C^2)^{-\frac{1}{2}}(I+C^2)^{-1}(I-C^2)^\frac{1}{2} \nonumber\\
~&=&2VC(I+C^2)^{-1}=V|B|.\nonumber
\end{eqnarray}
 Note that $I+A^*_0A_0=(I-AA^*)^{-1}$, and in the last equlity, we used the identity $\tanh(2x)=2\tanh x/(1+\tanh^2x)$. 

Now, we can define two biholomorphic automorphism $\psi_1$ and $\psi_2$ on $\mathcal{B}$, setting $\psi_1(X)=\psi_{A_0}(X)$ and $\psi_2(X)=\psi_{-T}(-X)$.
Indeed, set $\varphi=\psi_1\circ\psi_2$. Then we have
\begin{eqnarray}
\varphi(\hat{T})&=&\psi_1(\psi_2(\hat{T}))=\psi_{A_0}\psi_{-T}(-\hat{T})\nonumber\\
~&=&\psi_{A_0}(0)=-A=-\psi_{A_0}\psi_{-A_0}(A)=-\psi_{A_0}(B)=-\varphi(\hat{S}).\nonumber
\end{eqnarray}
It is clear that by lemma \ref{mobius}, $\varphi$ is a biholomorphic automorphism on $\mathcal{B}$.
\end{proof}
\begin{lem}\label{wot}
Under the conditions of lemmas \ref{phi} and \ref{mobius},
\begin{enumerate}
\item  The automorphisms $\varphi$ and $\psi_T$ are WOT-continuous.
\item The subset $\hat{D}$ of $\mathcal{B}$ is WOT-compact, for each admissible set $D$.
\end{enumerate}
\end{lem}
\begin{proof}
(1) The proof of lemma \ref{phi} and \ref{mobius} show that $\varphi$ and $\psi_T$ are of the form of \ref{linear transformation}. Hence by lemma\ref{llt}(iii), they are WOT-continuous. 
(2) To prove the second part, it is enough to show that this statement is true for all closed balls instead of the admissible subsets. 
We shall show first that if $D=\{T:~d(T,0)\leq r\}$ is a closed ball with center $0$, then $\hat{D}$ is WOT-compact. By theorem \ref{all metric}, for each $\hat{T}\in \hat{D}$ we have 
$$\tanh^{-1}\|\hat{T}\|=\omega(0,\|\hat{T}\|)=K_\mathcal{B}(\hat{T},0)=d(T,0)\leq r.$$
It follows that  $\|\hat{T}\|\leq (e^{2r}-1)/(e^{2r}+1)$. On the other hand, if $A\in L(K,H)$ and $\|A\|\leq (e^{2r}-1)/(e^{2r}+1)$ then $A=\hat{T}$ where $T=(1-A^*A)^{-\frac{1}{2}}A^*$, and $d(T,0)\leq r$. This inequality follows from the identity
$$d(T,0)=K_\mathcal{B}(\hat{T},\hat{0})=K_\mathcal{B}(A,0)=\omega(0,\|A\|)=\tanh^{-1}\|A\|.$$
We used the identity $\tanh^{-1}(e^{2r}-1)/(e^{2r}+1)=r$. So
$$\hat{D}=\{A\in L(K,H):~\|A\|\leq(e^{2r}-1)/(e^{2r}+1)\}.$$
The sentence follows from the fact that closed balls with center 0, are WOT-compact.

Now, let $D=\{T:~ d(T,T_1)\leq r\}$ where $T_1\in C(H,K)$. By lemma \ref{mobius}, there exists a biholomorphic automorphism $\psi$ on $\mathcal{B}$ with $\psi(\hat{T_1})=0$. Therefore
\begin{eqnarray}
\hat{D}&=&\{\hat{T}:~d(T,T_1)\leq r\}=\{\hat{T}:~K_\mathcal{B}(\hat{T},\hat{T_1})\leq r\}\nonumber\\
~&=&\{\hat{T}: ~K_\mathcal{B}(\psi(\hat{T}),\psi(\hat{T_1}))\leq r\}=\{\hat{T}:~K_\mathcal{B}(\psi(\hat{T}),0)\leq r\}\nonumber\\
~&=&\{\psi ^{-1}(\hat{S}):~K_\mathcal{B}(\hat{S},0)\leq r\}=\psi ^{-1}(\hat{D_1}),\nonumber
\end{eqnarray}
where $D_1=\{S:~d(S,0)\leq r\}$. The statement follows from the fact that $\psi ^{-1}$ is of the form of \ref{linear transformation} and by lemma \ref{llt}(iii), it is WOT-continuous. 
\end{proof}
\begin{lem} \label{mid}
For  $F= \{ T_1, \cdots , T_n \}$ of $n=2^k$ elements in $C(H,K)$, there is a $Q \in C(H,K)$ satisfying $d(Q,X) \leq \frac{1}{n} \sum_i d(T_i,X)$, for all $X\in C(H,K)$.
\end{lem}
\begin{proof}
First we note that for each $X\in C(H,K)$, $d(X,-X)=2d(X,0)$. Because, if $\hat{X}=V|\hat{X}|$ is the polar decomposition of $\hat{X}$, then
\begin{eqnarray}
d(X,-X)&=&\tanh ^{-1}\|L_X(-X)R_{-X}(X)^{-1} \|\nonumber\\
~&=&\tanh ^{-1}\|-2X^*[1+X(1+X^*X)^{-1}X^* ]^{-1}(1+XX^*)^{-\frac{1}{2}} \|\nonumber\\
~&=&\tanh ^{-1}\|2[1+X^*(1+XX^*)^{-1}X]^{-1}X^*(1+XX^*)^{-\frac{1}{2}} \|\nonumber\\
~&=&\tanh ^{-1}\|2\hat{X}[1+\hat{X}^*\hat{X}]^{-1} \|\nonumber\\
~&=&\tanh ^{-1}\| 2V |\hat{X}|(I+|\hat{X}|^2)^{-1}\|=\tanh ^{-1}\| 2 |\hat{X}|(I+|\hat{X}|^2)^{-1}\| \nonumber\\
~&=&\|\tanh ^{-1}( 2 |\hat{X}|(I+|\hat{X}|^2)^{-1})\|=\|2\tanh ^{-1}(|\hat{X}|)\| \nonumber\\
~&=&2\tanh ^{-1}\|\hat{X}\| =2d(X,0). \nonumber
\end{eqnarray} 
Notice that $V$ is a partial isometry, $(\ker V)^\perp=\overline{|\hat{X}|(K) }$, $| \hat{X}|$ is a positive operator. The function $f(x)=\tanh ^{-1}(x)$ is a non-decreasing on the interval $[0,\| \hat{X} \| ]$, which implies that $f(\|A\|)=\|f(A)\|$ for each $A\in L(K,H)$.

 Now, we will prove the  theorem by induction on $k$.\\
\underline{$1^{st}$ case}:  
let $T_1,T_2\in C(H,K)$. By Lemma \ref{phi}, there exists an isometric transformation $\varphi$, such that $\varphi(\hat{T_2})=-\varphi(\hat{T_1})$. Put $Q=(I-q^*q)^{-\frac{1}{2}}q^*$, where $q=\varphi^{-1}(0)$. So
\begin{eqnarray}
2d(Q,X)=2K_\mathcal{B}(q,\hat{X})&=&2K_\mathcal{B}(0,\varphi(\hat{X}))\nonumber\\
~&=&2d(0,Y)=d(Y,-Y)=K_\mathcal{B}(\varphi(\hat{X}),-\varphi(\hat{X})),\nonumber
\end{eqnarray}
where $Y=(I-\varphi(\hat{X})^*\varphi(\hat{X}))^{-\frac{1}{2}}\varphi(\hat{X})^*$.
Since, by theorem \ref{iso}, $A\mapsto -A$ is an isometric map with respect to the Kobayashi distance, we have
\begin{eqnarray}
2d(Q,X)&=&K_\mathcal{B}(\varphi(\hat{X}),-\varphi(\hat{X}))\nonumber\\
~&\leq & K_\mathcal{B}(\varphi(\hat{X}),\varphi(\hat{T_2}))+k_\mathcal{B}(\varphi(\hat{T_2}),-\varphi(\hat{X})) \nonumber\\
~&=& K_\mathcal{B}(\varphi(\hat{X}),\varphi(\hat{T_2}))+K_\mathcal{B}(-\varphi(\hat{T_2}),\varphi(\hat{X}))  \nonumber\\
~&=&  K_\mathcal{B}(\varphi(\hat{X}),\varphi(\hat{T_2}))+K_\mathcal{B}(\varphi(\hat{T_1}),\varphi(\hat{X}))  \nonumber\\
~&=& K_\mathcal{B}(\hat{X},\hat{T_2})+K_\mathcal{B}(\hat{T_1},\hat{X})  \nonumber\\
~&=& d(X,T_2)+d(T_1,X),\nonumber
\end{eqnarray}
for each $X\in C(H,K)$. \\
\underline{$2^{st}$ case}: Let  $F= \{ T_1, \cdots , T_n \}$ of $n=2^k$ elements in $C(H,K)$. For each $j=1,\dots, 2^{k-1}$, there exist $Q_j\in C(H,K)$ satisfying
$$d(Q_j,X)\leq \frac{1}{2}(d(T_{2j-1},X)+d(T_{2j},X)),$$
for all $X\in C(H,K)$. So, by hypothesis of induction, there  exist $Q\in C(H,K)$ satisfying
$$d(Q,X)\leq \frac{1}{2^{k-1}}\sum_j d(Q_j,X).$$
Indeed, we have
\begin{eqnarray}
d(Q,X)&\leq &\frac{1}{2^{k-1}}\sum_j \frac{1}{2}(d(T_{2j-1},X)+d(T_{2j},X))\nonumber\\
~&=&\frac{1}{2^k}\sum_i d(T_i,X).\nonumber
\end{eqnarray}  
\end{proof}
Now, we are ready to prove that the set of all admissible subsets of $C(H,K)$ is compact. Recall that $\mathcal{A}(M)$ is said to be compact if every descending chain of nonempty members of $\mathcal{A}(M)$ has nonempty intersection. 
\begin{cor}\label{comp}
$\mathcal{A}(C(H,K))$ is compact.
\end{cor}
\begin{proof} It is obvious, by the second part of lemma \ref{wot}.
\end{proof}
The next corollary follows from the proposition 5.1 of \cite{kham}. It says that each metric space $M$ for which $\mathcal{A}(M)$ is compact, is complete.  
\begin{cor}
The metric space $(C(H,K),d)$ is complete.
\end{cor}
\begin{cor}
The metric space $(C(H,K),d)$ is metrically convex.
\end{cor}
\begin{proof} 
Let $T_1, T_2\in C(H,K)$. By lemma \ref{mid}, there exists a $Q\in C(H,K)$  satisfying $d(Q,X)\leq \frac{1}{2}[d(T_1,X)+d(T_2,X)]$, for each $X\in C(H,K)$. This shows that $d(T_1,T_2)=d(T_1,Q)+d(Q,T_2)$. Hence the proof is completed.    
\end{proof}
\section{The normal structure and its applications}
We recall some definitions yielding to the notion of normal structure. Let $\mathcal{A}(M)$ denote the family of all admissible set in a metric space $(M,\rho)$. (Thus $D=\cap\{B:$ $B$ is a closed ball which contains $D\} $ if and only if $D\in\mathcal{A}(M)$.) We now follow the definition in the metric space \cite{kham}:
\begin{eqnarray}
r_a(D)&=&\sup \{\rho(a,x):~ x\in D\}, ~~~a\in M \nonumber\\
r(D)&=&inf\{r_x(D):~ x\in D\}.\nonumber
\end{eqnarray}
\begin{defn}
A point $a$ in an admissible subset $D$ of a metric space $(M,\rho)$ is called diametral if $r_a(D)=$diam$(D)$.
\end{defn}
\begin{lem} \label{seq0}
If all points of an admissible subset $D$ of metric space $C(H,K)$, are diametral and $\gamma$ be the diameter of $D$, then for each $F=\{T_1,\cdots ,T_n\}\subseteq D$ with $n=2^k$ and $\varepsilon >0$ there exists a point $C=C(F,\varepsilon)\in D$ such that
$$d(T_i,C)\geq (1-\varepsilon)\gamma,$$
for each $i=1,\cdots, n$.
\end{lem}
\begin{proof} By lemma \ref{mid}, there is a $Q$ such that $d(Q,X) \leq \frac{1}{n} \sum_i d(T_i,X)$, for all $X\in C(H,K)$. It is obvious that $Q\in D$. Since $Q$ is diametral, $r_Q(D)=\gamma$ and hence there is a $C\in D$ such that 
$$\gamma -\frac{\varepsilon}{n}\gamma\leq d(Q,C)\leq \gamma.$$
It follows that 
$$\gamma (1-\frac{\varepsilon }{n})\leq d(Q,C)\leq \frac{1}{n} \sum_i d(T_i,C).$$
If there exists a $j\leq n$ such that $d(T_j,C)<(1-\varepsilon)\gamma$, then 
$$\frac{1}{n} \sum_i d(T_i,C)<\frac{1}{n}[(1-\varepsilon)\gamma +(n-1)\gamma ]=(1-\frac{\varepsilon}{n})\gamma.$$
It is a contradiction, so $d(T_i,C)\geq (1-\varepsilon)\gamma,$ for each $i=1,\cdots, n$.
\end{proof}
The next theorem is one of the technical tools for proving results on normal structure for $(C(H,K),d)$.
\begin{thm} \label{seq1}
If all points of an admissible subset $D$ of metric space $C(H,K)$, are diametral and $\gamma$ be the diameter of $D$, then $D$ contains a net $\{T_\alpha\}_{\alpha\in I}$, such that
$\lim_\alpha d(T_\alpha,X)=\gamma$
for each $X\in D$.   
\end{thm}
\begin{proof} Let $I$ be the set of all pairs $\alpha=(F,\varepsilon)$, where $F=\{T_1,\cdots ,T_n\}\subseteq D$ with $n=2^k$ and $\varepsilon >0$. The set $I$ with the order given by 
\begin{center}
$\alpha_1=(F_1,\varepsilon_1) \leq \alpha_2=(F_2,\varepsilon_2)$ iff $F_1\subseteq F_2$ and $\varepsilon_1\geq \varepsilon_2,$
\end{center}
is a direct set. For $\alpha =(F,\varepsilon)$, we write $T_\alpha=C(F,\varepsilon)$ by lemma \ref{seq0}. 

It is clear that 
$$(1-\varepsilon)\gamma\leq d(T_\alpha ,X)\leq \gamma, $$
implies $\lim_\alpha d(T_\alpha,X)=\gamma$ for each $X\in D$.   
\end{proof}
\begin{defn}
A metric space $(M,\rho)$ has normal structure if $r(D)<$diam$(D)$ whenever $D\in \mathcal{A}(M)$ and diam$(D)>0$. 
\end{defn}
The assumption that $M$ has normal structure is equivalent to the following: If $D\in\mathcal{A} (M)$ consists of more than one point, then there is a number $r<$diam$(D)$ and a point $z\in D$ such that $D\subseteq B(z;r)$ \cite{kham}. So $M$ has normal structure if every admissible subset $D$ with more than one element, has a non-diametral point.  

The following is theorem 5.1 of \cite{kham}. We will use it, at the end of this section, to proving a fixed point theorem. This theorem is very useful to study the representation theory.  
\begin{thm}\label{thoper}
Suppose that a metric space $(M,\rho)$ has the normal structure and $\mathcal{A}(M)$ is compact. If a group of isometries of $M$ has a bounded orbit, then it has a fixed point. 
\end{thm}
\begin{thm}\label{normal}
The metric space $(C(H,K),d)$ has normal structure.
\end{thm}
\begin{proof} Assume the contrary, that is, all points in an admissible subset $D$ of $C(H,K)$ are diametral and $\gamma=$diam$(D)>0$. By theorem \ref{seq1}, $D$ contains a net $\{T_\alpha\}_{\alpha\in I}$, such that
$\lim_\alpha d(T_\alpha,X)=\gamma$, for each $X\in D$. The net $\{\hat{T}_\alpha \}_{\alpha \in I}$ contains a weakly convergent subnet, because by lemma \ref{wot}(2), $\hat{D}$ is WOT-compact. To simplify the notation we assume that $\{\hat{T}_\alpha \}_{\alpha \in I}$ itself is WOT-converges to some operator $\hat{P}\in\hat{D}$. By lemma \ref{mobius}, lemma \ref{llt} and theorem \ref{iso}, there exists an isometric biholomorphic automorphism $\psi$ on $(\mathcal{B}, K_\mathcal{B})$ with $\psi(\hat{P})=0$. Put $\hat{S}_\alpha=\psi (\hat{T}_\alpha )$, so
\begin{eqnarray}
\gamma &=&\lim_\alpha d(T_\alpha ,P)= \lim_\alpha K_\mathcal{B}(\hat{T}_\alpha ,\hat{P})\nonumber\\
~&=&\lim_\alpha K_\mathcal{B}(\psi(\hat{T}_\alpha) ,\psi(\hat{P}))=\lim_\alpha K_\mathcal{B}(\hat{S}_\alpha ,0)\nonumber\\
~&=&=\lim_\alpha d(S_\alpha ,0)=\lim_\alpha \tanh^{-1}\|L_{S_\alpha}(0)R_0(S_\alpha)^{-1} \|\nonumber\\
~&=&\lim_\alpha \tanh^{-1}\|-S_\alpha^*(1+S_\alpha S_\alpha^*)^{- \frac{1}{2}} \|=\lim_\alpha \tanh^{-1}\|\hat{S}_\alpha \|. \nonumber
\end{eqnarray}
Put $\lambda=\tanh \gamma $. We have $\lambda^2=\lim_\alpha \|\hat{S}_\alpha\|^2=\lim_\alpha\|\hat{S}_\alpha^*\hat{S}_\alpha\|$. Since $\{\hat{S}_\alpha^*\hat{S}_\alpha \}$ is a bounded net of bounded operators on finite dimensional Hilbert space $K$, we can assume that $\hat{S}_\alpha^*\hat{S}_\alpha\rightarrow \hat{q}$ in strong operator topology. It is clear that $\hat{q}$ is a positive bounded operator on $K$ and $\|\hat{q}\|=\lambda^2$. 

For each $\varepsilon>0$, fix $\beta$ with $\|\hat{S}_\beta^*\hat{S}_\beta -\hat{q}\|<\varepsilon$. By lemma \ref{mobius}, we have
\begin{eqnarray}
L_{S_\beta}(S_\alpha)R_{S_\alpha}(S_\beta)^{-1}&=&\psi_{S_\beta}(\hat{S_\alpha})\nonumber\\
~&=&(I-\hat{S}_\beta\hat{S}_\beta^*)^{-\frac{1}{2}}[\hat{S}_\alpha -\hat{S}_\beta][I-\hat{S}_\beta^* \hat{S}_\alpha]^{-1}(I-\hat{S}_\beta^*\hat{S}_\beta)^{\frac{1}{2}}.\nonumber
\end{eqnarray}

Since $\hat{S}_\beta^*$ is finite rank, $\hat{S}_\beta^*\hat{S}_\alpha\rightarrow 0$ in the norm topology. Hence
\begin{eqnarray}
\lim_\alpha \| L_{S_\beta}(S_\alpha)R_{S_\alpha}(S_\beta)^{-1} \| &=& \lim_\alpha \|(1-\hat{S}_\beta \hat{S}_\beta^*)^{- \frac{1}{2}}(\hat{S}_\beta-\hat{S}_\alpha)(1-\hat{S}_\beta^*\hat{S}_\beta)^{\frac{1}{2}} \| \nonumber\\
~&=& \lim_\alpha \| \hat{S}_\beta - (1-\hat{S}_\beta \hat{S}_\beta^*)^{- \frac{1}{2}}\hat{S}_\alpha (1-\hat{S}_\beta^*\hat{S}_\beta)^\frac{1}{2} \| \nonumber \\ 
~&=& \lim_\alpha \| \hat{S}_\beta - \hat{S}_\alpha (1-\hat{S}_\beta^*\hat{S}_\beta)^\frac{1}{2} \|. \nonumber
\end{eqnarray}

The second equality follows from the identity 
\vspace{0.3cm}

\leftline{$(1-\hat{S}_\beta \hat{S}_\beta^*)^{- \frac{1}{2}}(\hat{S}_\beta-\hat{S}_\alpha)(1-\hat{S}_\beta^*\hat{S}_\beta)^{\frac{1}{2}}+(1-\hat{S}_\beta \hat{S}_\beta^*)^{- \frac{1}{2}}\hat{S}_\alpha (1-\hat{S}_\beta^*\hat{S}_\beta)^\frac{1}{2}=$}
\rightline{$=(1-\hat{S}_\beta \hat{S}_\beta^*)^{- \frac{1}{2}}[(\hat{S}_\beta-\hat{S}_\alpha)+\hat{S}_\alpha](1-\hat{S}_\beta^*\hat{S}_\beta)^\frac{1}{2}=\hat{S}_\beta$,}
\vspace{0.2cm}
\leftline{and the last one holds since}
\vspace{0.3cm}
\leftline{$\|\hat{S}_\beta - (1-\hat{S}_\beta \hat{S}_\beta^*)^{- \frac{1}{2}}\hat{S}_\alpha (1-\hat{S}_\beta^*\hat{S}_\beta)^\frac{1}{2}-\hat{S}_\beta + \hat{S}_\alpha (1-\hat{S}_\beta^*\hat{S}_\beta)^\frac{1}{2}\|=$}
\vspace{-0.5cm}
\begin{eqnarray}
~&=&\| (1-\hat{S}_\beta \hat{S}_\beta^*)^{- \frac{1}{2}}\hat{S}_\alpha (1-\hat{S}_\beta^*\hat{S}_\beta)^\frac{1}{2}- \hat{S}_\alpha (1-\hat{S}_\beta^*\hat{S}_\beta)^\frac{1}{2}\|\nonumber\\
~&=&\| [(1-\hat{S}_\beta \hat{S}_\beta^*)^{- \frac{1}{2}}-I] \hat{S}_\alpha (1-\hat{S}_\beta^*\hat{S}_\beta)^\frac{1}{2}\|\nonumber\\
~&\leq & \| (1-\hat{S}_\beta \hat{S}_\beta^*)^{- \frac{1}{2}}-I\|\| \hat{S}_\alpha\|\| (1-\hat{S}_\beta^*\hat{S}_\beta)^\frac{1}{2}\|,\nonumber
\end{eqnarray}
and $\hat{S}_\alpha\rightarrow 0$ in WOT-topology. Therefore $\Lambda(\hat{S}_\alpha(x))\rightarrow 0$, for each $\Lambda\in H^*$ and  $x\in K$. By Hahn-Banach theorem,
$$\|\hat{S}_\alpha(x)\|=\sup\{|\Lambda(\hat{S}_\alpha(x))|:~\Lambda\in H^*, \|\Lambda\|\leq 1\}.$$
Indeed, for each $n$, there exists $\Lambda_n\in H^*$ such that
$$\|\hat{S}_\alpha(x)\|<|\Lambda_n(S_\alpha(x))|+\frac{1}{n}.$$
This implies that $\lim_\alpha\|\hat{S}_\alpha(x)\|<\frac{1}{n}$, for each $n$. Therefore $\lim_\alpha\|\hat{S}_\alpha(x)\|=0$, and hence $\lim_\alpha\|\hat{S}_\alpha\|=0$. Because $K$ is finite dimensional  and
$$\|\hat{S}_\alpha\|\leq\sum_i\|\hat{S}_\alpha(e_i)\|.$$
where $K=\langle e_1,\dots,e_l\rangle$.  Also, the equality
\vspace{0.3cm}

\leftline{ $\lim_\alpha \| [ \hat{S}_\beta - \hat{S}_\alpha (1-\hat{S}_\beta^*\hat{S}_\beta)^\frac{1}{2}]^*[\hat{S}_\beta - \hat{S}_\alpha(1-\hat{S}_\beta^*\hat{S}_\beta)^\frac{1}{2}] \| $ }
\rightline{$= \lim_\alpha \|  \hat{S}_\beta^*\hat{S}_\beta + (1-\hat{S}_\beta^*\hat{S}_\beta)^\frac{1}{2} \hat{S}_\alpha^*\hat{S}_\alpha(1- \hat{S}_\beta^*\hat{S}_\beta)^\frac{1}{2} \|$,} 
\vspace{0.2cm}
holds, because $\hat{S}_\alpha (1-\hat{S}_\beta^*\hat{S}_\beta)^{\frac{1}{2}}\rightarrow 0$ in WOT-Topology.
Hence, same as befor, the net
$$\hat{S}_\beta^* \hat{S}_\alpha (1-\hat{S}_\beta^*\hat{S}_\beta)^{\frac{1}{2}} +[\hat{S}_\alpha (1-\hat{S}_\beta^*\hat{S}_\beta)^{\frac{1}{2}}]^* \hat{S}_\beta $$
tends to zero in norm topology. Furthermore,
$$(1-\hat{S}_\beta^*\hat{S}_\beta)^\frac{1}{2} \hat{S}_\alpha^*\hat{S}_\alpha(1- \hat{S}_\beta^*\hat{S}_\beta)^\frac{1}{2}$$
tends in norm topology to $(1-\hat{S}_\beta^*\hat{S}_\beta)^\frac{1}{2} \hat{q}(1- \hat{S}_\beta^*\hat{S}_\beta)^\frac{1}{2}$. Therefore 
$$\lim_\alpha \| L_{S_\beta}(S_\alpha)R_{S_\alpha}(s_\beta)^{-1} \|^2= \|  \hat{S}_\beta^*\hat{S}_\beta + (1-\hat{S}_\beta^*\hat{S}_\beta)^\frac{1}{2} \hat{q}(1- \hat{S}_\beta^*\hat{S}_\beta)^\frac{1}{2} \|.  $$
On the other hand, since $\|\hat{S}_\beta^*\hat{S}_\beta -\hat{q}\|<\varepsilon$, we have
\vspace{0.3cm}

\leftline{$\| \hat{S}_\beta^*\hat{S}_\beta + (1- \hat{S}_\beta^*\hat{S}_\beta)^{\frac{1}{2}}\hat{q}(1-\hat{S}_\beta^*\hat{S}_\beta)^{\frac{1}{2}}-[\hat{S}_\beta^*\hat{S}_\beta+(1-\hat{S}_\beta^*\hat{S}_\beta)\hat{S}_\beta^*\hat{S}_\beta]  \|=$}
\vspace{-0.5cm}
\begin{eqnarray}
~&\leq &\|(1- \hat{S}_\beta^*\hat{S}_\beta)^{\frac{1}{2}}\| ^2\|\hat{q}-hat{S}_\beta^*\hat{S}_\beta\|\nonumber\\
~&\leq &2\|\hat{q}-hat{S}_\beta^*\hat{S}_\beta\|<2\varepsilon.\nonumber
\end{eqnarray}
The inequalities $\lambda^2-\varepsilon<\|\hat{S}_\beta^*\hat{S}_\beta\|<\lambda^2+\varepsilon$ imply
\vspace{0.3cm}

\leftline{$\|\hat{S}_\beta^*\hat{S}_\beta+(1-\hat{S}_\beta^*\hat{S}_\beta)\hat{S}_\beta^*\hat{S}_\beta \|=$}
\vspace{-0.5cm}
\begin{eqnarray}
~&= & 2\|\hat{S}_\beta^*\hat{S}_\beta\|-\|\hat{S}_\beta^*\hat{S}_\beta\|^2\nonumber\\
~&\geq & 2(\lambda^2 -\varepsilon)-(\lambda^2+\varepsilon)^2\nonumber\\
~&\geq &\lambda^2 +2\varepsilon,\nonumber
\end{eqnarray}
the last inequality is satisfied if $\varepsilon$ is sufficiently small ($|\lambda|<1$). So

$$\lim_\alpha \| L_{S_\alpha}(S_\beta)R_{S_\beta}(S_\alpha)^{-1} \|^2>\lambda^2 +2\varepsilon-2\varepsilon=\lambda^2. $$
This implies that $\lim_\alpha d(T_\beta ,T_\alpha)=\lim_\alpha d(S_\beta , S_\alpha)>\gamma=$diam$(D)$. 

So we get a contradiction and hence the metric space $(C(H,K),d)$ has normal structure. 
\end{proof}
\begin{cor}
If a group of isometries of $(C(H,K),d)$ has a bounded orbit, then it has a fixed point. 
\end{cor}
\begin{proof} It follows immediately from theorems \ref{normal}, \ref{thoper}, and corollary \ref{comp}. 
\end{proof}

\begin{defn}\label{hbi}
A map $\varphi$ on $C(H,K)$ is called h-biholomorphic if there exists a biholomorphic map $\psi$ on $\mathcal{B}$ such that $\psi(\hat{X})=\widehat{\varphi (X)}$, for all $X\in C(H,K)$.  
\end{defn}
\begin{ex}
Let $\eta$ be as in (\ref{linear transformation}). We define a transformation $\varphi$ on $C(H,K)$ setting $\varphi(X)=(1-\eta(\hat{X})^*\eta(\hat{X}))^{-\frac{1}{2}}\eta(\hat{X})^*$. It is obvious that $\eta(\hat{X})=\widehat{\varphi(X)}$, for all $X\in C(H,K)$. Hence, $\varphi$ is a h-biholomorphic automorphism on $C(H,K)$. 
\end{ex}
\begin{lem}\label{hbiiso}
Let $\varphi$ be a h-biholomorphic automorphism on $C(H,K)$, then 
$$d(\varphi(T),\varphi(S))=d(T,S),$$
for all $T,S \in C(H,K)$.
\end{lem}
\begin{proof}
By definition \ref{hbi}, there is a biholomorphic map $\psi$ on $\mathcal{B}$ such that $\psi(\hat{X})=\widehat{\varphi(X)}$, for all $X\in C(H,K)$. Now, theorem \ref{iso} implies that 
$$d(\varphi(T),\varphi(S))=K_\mathcal{B}(\widehat{\varphi(T)},\widehat{\varphi(S)})=K_\mathcal{B}(\psi(\hat{T}),\psi(\hat{S}))=K_\mathcal{B}(\hat{T},\hat{S})=d(T,S).$$
\end{proof}

Since, by lemma \ref{hbi}, h-biholomorphic mappings of $C(H,K)$ are isometric, we also get the following corollary.

\begin{cor}
If a group of h-biholomorphic automorphisms of $C(H,K)$ has at least one bounded orbit, then it has a fixed point.
\end{cor}


\end{document}